\documentclass{amsart}

\usepackage{pinlabel,graphicx,epstopdf}

\title{Volume estimates for equiangular hyperbolic Coxeter polyhedra} 
\author{Christopher K. Atkinson} 
\address{Department of Mathematics,
Statistics, and Computer Science, University of Illinois at Chicago, 851 S.
Morgan St., Chicago, IL
60607}
\email{atkinson@math.uic.edu} 
\urladdr{http://www.math.uic.edu/~atkinson}

\newtheorem{theorem}{Theorem}[section] 
\newtheorem{lemma}[theorem]{Lemma}
\newtheorem{corollary}[theorem]{Corollary}

\newtheorem{proposition}[theorem]{Proposition}

\newtheorem*{acknowledgments}{Acknowledgments}

\newcommand{\ZZ}{\mathbb{Z}} 

\newcommand{\QQ}{\mathbb{Q}}
 
\newcommand{\RR}{\mathbb{R}}

\newcommand{\DD}{\Delta}
\newcommand{\GG}{\Gamma}

\def\vol{\mbox{\rm{Vol}}} 

\def\area{\mbox{\rm{Area}}}

\def\N{\mbox{\rm{N}}}

\def\Vol{\mbox{\rm{Vol}}} 
\def\area{\mbox{\rm{Area}}}
\def\min{\mbox{\rm{min}}}

\def\deg{\mbox{\rm{deg}}}

\def\HH{\mathbb{H}}
\def\AA{\mathcal{A}}
 
\def\FF{\mathcal{F}}

\def\PP{\mathcal{P}}
\def\AA{\mathcal{A}}
\def\WW{\mathcal{W}}
\def\BB{\mathcal{B}}
\def\GG{\mathcal{G}}
\def\TT{\mathcal{T}} 
\def\aa{\alpha} 
\def\QQ{\mathcal{Q}}
\def\RRR{\mathcal{R}}

\begin{document}

\begin{abstract}  An equiangular hyperbolic Coxeter polyhedron is a hyperbolic
polyhedron where all dihedral angles are equal to $\pi/n$ for some fixed
$n\in\ZZ$, $n\geq 2.$  It is a consequence of Andreev's theorem that either
$n=3$ and the polyhedron has all ideal vertices or that $n=2$.  Volume estimates
are given for all equiangular hyperbolic Coxeter polyhedra.  
\end{abstract}

\maketitle

\section{Introduction}

An orientable $3$--orbifold, $\QQ$, is determined by an underlying $3$--manifold,
$X_{\QQ}$, and a trivalent graph, $\Sigma_{\QQ}$, labeled by integers.  If $\QQ$
carries a hyperbolic structure then it is unique by Mostow
rigidity, so the hyperbolic volume of $\QQ$ is an invariant of $\QQ$.
Therefore, for hyperbolic orbifolds 
with a fixed underlying manifold, the volume is a function of the
labeled graph $\Sigma.$  In this paper, methods
for estimating the volume of orbifolds of a restricted type in terms of this
labeled graph will be described.

The orbifolds studied in this paper are quotients of $\HH^3$ by reflection
groups generated by reflections in hyperbolic Coxeter polyhedra.  A Coxeter
polyhedron is one where each dihedral angle is of the form $\pi/n$ for some
$n\in \ZZ$, $n\geq 2$.  Given a hyperbolic Coxeter polyhedron $\PP$, consider
the group generated by reflections through the geodesic planes determined by its
faces, $\Gamma(\PP)$.  Then $\Gamma(\PP)$ is a Kleinian group which acts on
$\HH^3$ with fundamental domain $\PP$.  The quotient, $\mathcal{O}=\HH^3/
\Gamma(\PP)$, is a
non--orientable orbifold with singular locus $\PP^{(2)}$, the $2$--skeleton of
$\PP$.  One may think of obtaining $\mathcal{O}$ by ``mirroring" the faces of
$\PP$.  It is a consequence of Andreev's theorem that any equiangular hyperbolic Coxeter polyhedron
has either all dihedral angles equal to $\pi/3$ and is ideal or has all dihedral
angles equal to $\pi/2$ \cite{andreev1, andreev2}.  This paper gives two-sided,
combinatorial volume estimates for all equiangular
hyperbolic Coxeter polyhedra.  In this paper, only polyhedra with finite volume
will be considered.
 
Lackenby gave volume estimates for hyperbolic alternating link complements in
\cite{lackenby} in terms of the twist number of the link.  His work was part of
what led to this investigation of how geometric data arises from associated
combinatorial data.  Some of the techniques used
in this paper follow methods used by Lackenby.  The lower bound given by Lackenby was improved by
Agol, Storm, and Thurston in \cite{agolhaken}.  In his thesis \cite{inoue},
Inoue has identified the two smallest-volume, compact, right-angled hyperbolic
polyhedra.  He also gave a method to
order such polyhedra based on a decomposition into L\"{o}bell polyhedra,
provided that the volume of any given right-angled polyhedron can be calculated
exactly.

The results of this paper can be used to list all equiangular hyperbolic
polyhedra with volume not exceeding some fixed value.  A sample application of
this is to classify all arithmetic Kleinian maximal reflection groups.  Agol has
shown in \cite{agolreflect} that the number of such groups is finite up to
conjugacy.  Given a maximal reflection group $\Gamma$ generated by reflections
in a polyhedron $\PP$, he gives an upper bound, independent of $\Gamma$, for the
volume of $\PP/\Theta$ where $\Theta$ is the group of symmetries of $\PP$ which
are not reflections. One could therefore
attempt to classify such groups using the results of this paper by writing down a
list of all  polyhedra of sufficiently small volume and checking arithmeticity
for those for which the quotient by additional symmetries has small enough
volume.

\begin{acknowledgments}  The author would like to thank his thesis advisor, Ian
Agol, for his excellent guidance and the referee for many valuable comments.
The author was partially supported by NSF grant DMS-0504975.
\end{acknowledgments}

\section{Summary of results}

The results of the paper are outlined in this section.  Theorems \ref{sharp2},
\ref{pi2finite}, and \ref{pi2general} concern the volumes of right-angled
hyperbolic polyhedra.  Theorem \ref{sharp3} concerns hyperbolic polyhedra with
all angles $\pi/3$.  Before stating any results, some terminology will be
introduced.

An \textit{abstract polyhedron} is a cell complex on $S^2$ which can be realized
by a convex Euclidean polyhedron.  A theorem of Steinitz says that realizability
as a convex Euclidean polyhedron is equivalent to the $1$--skeleton of the cell
complex being $3$--connected \cite{steinitz}.  A graph is $3$--connected if the removal of any $2$
vertices along with their incident edges leaves the complement connected.
 A \textit{labeling} of an abstract polyhedron $P$ is a map
$$\Theta : \mbox{Edges}(P) \to (0,\pi/2].$$ 
For an abstract polyhedron, $P$, and a labeling, $\Theta$, the pair
$(P,\Theta)$ is a \textit{labeled abstract polyhedron.}   A labeled abstract
polyhedron is said to be \textit{realizable as a hyperbolic polyhedron} if there
exists a hyperbolic polyhedron, $\PP$, such that there is a label-preserving
graph isomorphism between $\PP^{(1)}$ with edges labeled by dihedral angles and $P$
with edges labeled by $\Theta$.  A \textit{defining plane} for a hyperbolic
polyhedron $\PP$ is a hyperbolic plane $\Pi$ such that $\Pi \cap \PP$ is a face
of $\PP$.  A labeling $\Theta$ which is constantly equal
to $\pi/n$ is \textit{$\pi/n$--equiangular}.  Suppose $G$ is a graph and $G^*$ is
its dual graph.  A \textit{$k$--circuit} is a simple closed curve composed of $k$
edges in $G^*$.  A \textit{prismatic $k$--circuit} is a $k$--circuit $\gamma$ so
that no two edges of $G$ which correspond to edges traversed by $\gamma$ share a
vertex.  The following theorem is a special case of Andreev's Theorem, which
gives necessary and sufficient conditions for a labeled abstract polyhedron to be
realizable as a hyperbolic polyhedron.

\begin{theorem}[Andreev's theorem for $\pi/2$--equiangular polyhedra]
A $\pi/2$--equiangular labeled abstract polyhedron $(P,\Theta)$ is realizable as a
hyperbolic polyhedron, $\PP$, if and only if the following conditions hold:
\begin{enumerate}
\item $P$ has at least $6$ faces.
\item $P$ each vertex has degree $3$ or degree $4$.
\item For any triple of faces of $P$, $(F_i,\, F_j,\, F_k),$ such that $F_i \cap
F_j$ and $F_j \cap F_k$ are edges of $P$ with distinct endpoints, $F_i \cap F_k
= \emptyset.$
\item $P^*$ has no prismatic $4$ circuits.
\end{enumerate}
Furthermore, each degree $3$ vertex in $P$ corresponds to a finite vertex in
$\PP$, each degree $4$ vertex in $P$ corresponds to an ideal vertex in $\PP$,
and the realization is unique up to isometry.
\end{theorem}

The first result gives two-sided volume estimates for ideal, $\pi/2$--equiangular
hyperbolic polyhedra.

\begin{theorem} \label{theorempi2} \label{sharp2} If $\PP$ is an 
ideal $\pi/2$--equiangular
polyhedron with $N$ vertices, then 
$$ (N-2) \cdot \frac{V_8}{4}  \leq \vol(\PP)\leq
(N-4) \cdot \frac{V_8}{2}, $$ 
where $V_8$ is the volume of a regular ideal
hyperbolic octahedron.  Both inequalities are equality when $\PP$ is the
regular ideal hyperbolic octahedron.    
 There is a sequence of ideal $\pi/2$--equiangular
polyhedra $\PP_i$ with $N_i$ vertices such that $\vol(\PP_i)/N_{i}$ approaches
$V_8/2$ as $i$ goes to infinity.
\end{theorem}

The constant $V_8$ is the volume of a regular ideal hyperbolic octahedron.  In
terms of the Lobachevsky function, 
$$\Lambda(\theta)=-\int_0^{\theta} \log |2\sin\,t|\, dt,$$
 $V_8=8\Lambda(\pi/4).$  This
volume can also be expressed in terms of Catalan's constant, $K$, as $V_8 = 4K$,
where
$$K=\sum_{n=0}^{\infty} \frac{(-1)^n}{(2n+1)^2}.$$
The value of $V_8$ to five decimal places is $3.66386.$

The proof is spread throughout the paper.  The lower bound will be shown
in Section ~\ref{S:lower} to be a consequence of the stronger
Theorem~\ref{pi2lower}, which depends also on information about the number of
faces.   The upper bound in Theorem~\ref{theorempi2} will be proved in Section
~\ref{S:upper} and will be shown to be asymptotically sharp in Section
~\ref{S:seq}.

Using similar techniques, the following theorem giving volume estimates for
compact $\pi/2$--equiangular polyhedra will be proved:

\begin{theorem}
\label{pi2finite} 
If $\PP$ is a compact $\pi/2$--equiangular hyperbolic polyhedron with $N$
vertices,
then 
$$ (N-8) \cdot \frac{V_8}{32} \leq \vol(\PP) < (N-10)\cdot \frac{5V_3}{8},$$
where $V_3$ is the volume of a regular ideal hyperbolic tetrahedron.  
There is a sequence of compact polyhedra, $\PP_i$, with $N_i$ vertices such that
$\vol(\PP_i)/N_i$ approaches $5V_3/8$ as $i$ goes to infinity.
\end{theorem}

In terms of the Lobachevsky function, $V_3=2 \Lambda(\pi/6).$  To five decimal
places, $V_3$ is $1.01494$.  

Combining the methods of Theorems~\ref{theorempi2} and \ref{pi2finite},
estimates will be given for $\pi/2$--equiangular polyhedra with both finite and
ideal vertices:
  
\begin{theorem} 
\label{pi2general} If $\PP$ is a $\pi/2$--equiangular
hyperbolic polyhedron with $N_{\infty}$ ideal
vertices and $N_F$ finite vertices, then 
$$\frac{4N_{\infty}+N_F-8}{32} \cdot V_8 \leq \vol(\PP) < (\N_{\infty}-1)\cdot \frac{V_8}{2} +
 N_F\cdot \frac{5 V_3}{8} .$$ 
\end{theorem}

The proofs of the lower bounds in Theorems~\ref{pi2finite} and \ref{pi2general} appear in
Section \ref{S:lower}.  The upper bounds will be proved in Section
\ref{S:upper}.

Two-sided volume estimates for $\pi/3$--equiangular polyhedra are also given.
First, the special case of Andreev's Theorem for $\pi/3$--equiangular polyhedra
is stated.

\begin{theorem}[Andreev's theorem for $\pi/3$--equiangular polyhedra]\label{and3}
A $\pi/3$--equiangular abstract polyhedron, $(P, \Theta)$, is realizable as a
hyperbolic polyhedron, $\PP$ if each vertex of $P$ has degree $3$ and $P^*$ has
no prismatic $3$--circuits.  Furthermore, each vertex of $\PP$ is ideal and $\PP$
is unique up to isometry.  
\end{theorem}

The next theorem gives analogous results to that of
Theorem~\ref{theorempi2} for ideal $\pi/3$--equiangular polyhedra.

\begin{theorem} \label{sharp3} \label{theorempi3} If $\PP$ is an ideal
$\pi/3$--equiangular polyhedron
with $N>4$ vertices, then 
$$N\cdot\frac{V_3}{3}  
\leq \vol(\PP) \leq (3N-14) \cdot \frac{V_3}{2},$$
 where $V_3$ is the volume of a
regular ideal hyperbolic tetrahedron.  The upper bound is sharp for the regular
ideal hyperbolic cube.  
There is a sequence of ideal $\pi/3$--equiangular
polyhedra $\PP_i$ with $N_i$ vertices such that $\vol(\PP_i)/N_{i}$ approaches
$3V_3 / 2$ as $N_i$ increases to infinity.  
\end{theorem}

The lower bound in this theorem will be proved in Section \ref{S:lower3} by packing
horoballs around the vertices.  The upper bound will be proved in Section
\ref{S:upper} and will be shown to be asymptotically sharp in Section
\ref{S:seq}, mirroring the proofs in the $\pi/2$ case.  In a personal
communication, Rivin has indicated how to improve the lower bound to $N \cdot
\frac{3 V_3}{8}$.  His argument will be briefly described at the end of Section
\ref{S:lower3}.

\section{Lower volume bound for ideal $\pi/2$--equiangular polyhedra}\label{S:lower} 
The key result used in proving
the lower volume bound in Theorems~\ref{theorempi2} and \ref{pi2lower} is a
theorem of Miyamoto which says that the volume of a complete hyperbolic
$3$--manifold with totally geodesic boundary is greater than or equal to a
constant multiple of the area of the boundary \cite{miyamoto}:

\begin{theorem}[Miyamoto] \label{miyamoto}  If $\mathcal{O}$ is a complete
hyperbolic $3$--orbifold with non-empty totally geodesic boundary, then
 $$\vol(\mathcal{O})\geq \area(\partial
\mathcal{O}) \cdot \frac{V_8}{4 \pi}$$ 
with equality only if $M$ can be decomposed into
regular ideal hyperbolic octahedra.
\end{theorem}

Miyamoto actually only stated this theorem for manifolds.  The orbifold version
follows immediately, however.  Given a complete hyperbolic $3$--orbifold, $\mathcal{O}$, with non-empty totally
geodesic boundary, Selberg's lemma implies the existence of an integer $m$ such
that an $m$--fold cover of $\mathcal{O}$ is a manifold $M$
\cite{selberg}.  Then since $M$ is a
finite cover, $\vol (M)=m \cdot \vol (\mathcal{O})$ and $\area
(\partial M) = m \cdot \area (\partial \mathcal{O})$.

Theorem~\ref{theorempi2} is a consequence of the following
stronger theorem which also takes into account information about the faces of
the polyhedron.  Andreev's theorem for $\pi/2$--equiangular polyhedra implies
that the $1$--skeleton of an ideal $\pi/2$--equiangular polyhedron is a
four-valent graph on $S^2$.  The faces, therefore, can be partitioned into a
collection of black faces, $\BB,$ and a collection of white faces, $\WW$, so
that no two faces of the same color share an edge.  Denote by $|\BB|$ and
$|\WW|$ the number of black and white faces, respectively.

\begin{theorem}\label{pi2lower}   
Suppose $\PP$ is an ideal $\pi/2$--equiangular polyhedron with $N$ vertices.
If $\PP$ is $2$--colored with $|\BB|\geq|\WW|$, then 
$$ (N-|\WW|) \cdot \frac{V_8}{2} \leq
\vol{(\PP)}.$$ 
This inequality is sharp for an infinite family of polyhedra
obtained by gluing together regular ideal hyperbolic octahedra.  
\end{theorem} 

\begin{proof}
Given an ideal $\pi/2$--equiangular
polyhedron $\PP$, consider the orbifold $\HH^3 / \Gamma(\PP)$, where
$\Gamma(\PP)$ is the reflection group generated by $\PP$, as described in the
introduction.  Denote the generators of $\Gamma(\PP)$ by the same symbol
denoting the face of
$\PP$ through which it is a reflection.  If $A$ is a face of $\PP$, let
$\Gamma_{A}(\PP)$ be the group obtained from $\Gamma(\PP)$ by removing the
generator $A$ and all relations involving $A$.  This may be thought of as
``un-mirroring" the face $A$.  Then $CC(\HH^3/\Gamma_{A}(\PP)),$ the convex core
of $\HH^3/\Gamma_{A}(\PP)$, is an orbifold with totally geodesic boundary and
the same volume as $\PP$.  Note that the face $A$ is a suborbifold of
$\HH^3 / \Gamma(\PP)$ because all dihedral angles of $\PP$ are $\pi/2$.  In
general, a face of a polyhedron meeting a dihedral angle not equal to $\pi/2$ will not be
a suborbifold, and removing the generators and relations corresponding to that
face from the reflection group will not give a totally geodesic boundary because
the preimage of that face in $\HH^3$ will not be a collection of disjoint
geodesic planes.  

To get the best lower bound on volume from Theorem~\ref{miyamoto}, the boundary
should be chosen to have the greatest possible area.  Given a collection of
faces $\AA =\bigcup_{i=1} ^M A_i$, denote the group obtained by removing all
generators and relations involving the $A_i$ by $\Gamma_{\AA} (\PP)$.  If no two
faces in such a collection $\AA$ share an edge, the orbifold $CC(\HH^3 /
\Gamma_{\AA} (\PP))$ has totally geodesic boundary.

As described before the statement of the theorem, $2$--color the faces of $\PP$
black and white so that no two faces of the same color share an edge.
Suppose that coloring is chosen so that the number of black faces, $|\BB|$, is at
least the number of white faces, $|\WW|$.  This choice will ensure that the sum
of the areas of the faces in 
$\WW$ is at least the sum of the areas of the black faces, as will be seen in
Lemma~\ref{areaw}.  Then,
$CC(\HH^3 /  \Gamma_{\WW}(\PP))$ is an orbifold with totally geodesic boundary
consisting of the $W_i$.  Theorem~\ref{miyamoto} applied to this orbifold gives
$$\vol(\PP) \geq \area(\WW) \cdot \rho_3 (0).$$
The following lemma gives the area of $\WW$ and completes the proof of the lower
bound.

\begin{lemma}\label{areaw}
With $\WW$ as above, 
$$\area(\WW)=2\pi(N-|\WW|).$$
\end{lemma}
\begin{proof}
Consider $D\WW$, the double of $\WW$ along its
boundary. Recall that the orbifold Euler characteristic of a $2$--orbifold $Q$ is
$$\chi(Q)=\chi(X_Q) - \sum_{i} (1-1/m_i),$$ where $X_Q$ is the underlying
topological space and $Q$ has cone points of orders $m_i$ \cite{CHK}.   The
faces in $\WW$ meet each of the $N$ vertices and each of the $E$ edges of $\PP$.
The orbifold $D\WW$ is a union over all faces of $\WW$ of doubled ideal
polygons.  Each of these doubled ideal polygons is a $2$--sphere with a cone
point of order $\infty$ for each vertex of the polygon.
 Each vertex of $\PP$ contributes
a cone point to two of these doubled polygons.  Therefore
$$\chi(D\WW)= |\WW| \cdot \chi(S^2) - 2N=2(|\WW|-N).$$ 
The punctured surface $D\WW$ is hyperbolic, being a
union of hyperbolic $2$--orbifolds.
Therefore, the Gauss--Bonnet theorem for orbifolds implies that
$$\area(D\WW)=-2\pi \chi(D\WW)=4\pi (N-|\WW|),$$ 
so that
$\area(\WW)=2\pi(N-|\WW|)$.  
\end{proof}

The polyhedra which realize the lower bound as claimed in Theorem~\ref{pi2lower}
are constructed by gluing together octahedra.  Consider a regular
ideal hyperbolic octahedron with faces colored white and black, so that no two faces
of the same color share an edge.  For a single octahedron, $N=6$ and
$|\WW|=4$, so the lower bound is equal to the volume.   To obtain an infinite
number of polyhedra which satisfy the claim, glue a finite collection of
$2$--colored regular ideal hyperbolic
octahedra together, only gluing black faces to black faces.  Each successive
gluing results in a polyhedron with $3$ more vertices and $1$ more white face.  Therefore by induction, for each example constructed in this fashion,
the lower inequality in Theorem~\ref{pi2lower} will be equality.  See
figure~\ref{figure:gluing}.  Note that
gluing octahedra in a different pattern than described yields examples which do
not satisfy the claim.   This completes the proof of Theorem~\ref{pi2lower}.
\end{proof}

\begin{figure} 
\begin{center}
\scalebox{0.475}{\includegraphics{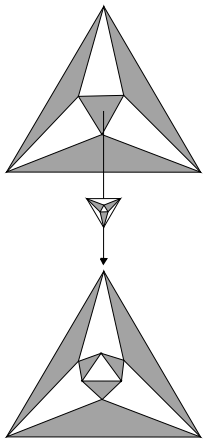}}
\end{center}
\caption{An example of the gluing.} 
\label{figure:gluing}
\end{figure}

The argument giving the lower bound in Theorem~\ref{theorempi2} as a consequence
of Theorem~\ref{pi2lower} is similar to the proof of Theorem $5$ in
\cite{lackenby}.  The idea is to average the estimate coming from the black faces
with the estimate coming from the white faces.  Consider a sort of dual polyhedron, $\GG$, to $\PP$.  The
vertices of $\GG$ are the white faces in the specified coloring of $\PP$.  For any two faces of $\PP$ which share a vertex, the corresponding
two vertices in $\GG$ are connected by an edge.  The $2$--skeleton is homeomorphic
to $S^2$, so the Euler characteristic of $\GG^{(2)}$ is $2$.  The number of
vertices, edges and faces of $\GG^{(2)}$ respectively are $|\WW|$, $N$, and
$|\BB|$, where $N$ is the number of vertices of $\PP$.  Hence $|\BB|= 2-|\WW|
+N$.  An application of Lemma~\ref{areaw} yields $\area(\BB)=2\pi
(N-|\BB|)=2\pi(|\WW|-2)$.  Hence 
$$\vol(\PP) \geq (|\WW|-2) \cdot \frac{V_8}{2}.$$  
Therefore, combining this
inequality with the inequality from Theorem~\ref{pi2lower}, $$\vol(\PP)
\geq (N-2) \cdot \frac{V_8}{4},$$ proving the lower bound of Theorem~\ref{theorempi2}.   

The lower bound in Theorem~\ref{pi2general}, where $\PP$ is a
$\pi/2$--equiangular polyhedron with vertices which are either finite or
ideal, is proved similarly.
\begin{proposition}
Suppose $\PP$ is a $\pi/2$--equiangular hyperbolic polyhedron with $N_{\infty}$ ideal vertices, $N_F$
finite vertices, and $|\FF|$ faces.  Then
$$\frac{8N_{\infty}+3N_F-4|\FF|}{32} \cdot V_8 \leq \vol(\PP).$$
\end{proposition}
\begin{proof}
By the four color theorem, a $4$--coloring of the faces of $\PP$ may be found \cite{appelhaken,gonthier}.
One of the collections of faces of the same color, say $\BB$, has area at least
$\area(\partial\PP)/4$, where $\area(\partial \PP)$ should be interpreted as the
sum of the areas of all the faces of $\PP$.  The area of a hyperbolic $k$--gon
with interior angles summing to $S$ is $(k-2)\pi - S$.  The sum of the interior
angles of a face of $\PP$ is $n_F \cdot \pi/2,$ where $n_F$ is the number of
finite vertices of the face.  Hence the area of a single
face is $$\pi \left(n_{\infty}+\frac{n_F}{2}-2\right),$$
where $n_{\infty}$ is
the number of ideal vertices of the face.  Summing over all faces and using the
fact that each finite vertex is a vertex of three faces and each ideal vertex is
a vertex of four faces gives
$$\area(\partial \PP) = \pi \cdot \frac{8N_{\infty}+3N_F -4 |\FF|}{2}.$$
Then since $\area(\BB) \geq \area(\partial \PP)/4$,
 applying Miyamoto's theorem to $\BB$ finishes the proof of the proposition.
\end{proof}

The lower bound in Theorem~\ref{pi2finite} follows by setting $N_{\infty}=0$.


\section{Lower volume bound for ideal $\pi/3$--equiangular polyhedra}\label{S:lower3}

In this section the lower bound given in Theorem~\ref{theorempi3} will be
proved:

\begin{proposition}\label{pi3prop} If $\PP$ is an ideal $\pi/3$--equiangular
polyhedron with $N>4$ vertices, then 
$$\vol(\PP) > N \cdot \frac{V_3}{3} .$$
\end{proposition}
Before proving this proposition, two lemmas about ideal $\pi/3$--equiangular
polyhedra with more than $4$ vertices are needed.
 Consider a $\pi/3$--equiangular polyhedron, $\PP,$ in the upper half-space
model for $\HH^3$ with one vertex placed at the point at infinity.  The link of each vertex
is Euclidean, so must be an equilateral Euclidean triangle since all dihedral
angles are $\pi/3$.  Thus the image of $\PP$ under the orthogonal projection to
the bounding plane of $\HH^3$ is an equilateral triangle.  This
triangle will be referred to as the {\it base triangle}.  The three vertices
adjacent to the vertex at infinity will be called {\it corner vertices}.  

The following is a corollary of Andreev's theorem for $\pi/3$--equiangular
polyhedra.

\begin{corollary}\label{cor4edges}
If $\PP$ is a $\pi/3$--equiangular polyhedron which has more than $4$ vertices,
then each face of $\PP$ has at least $4$ edges.
\end{corollary}

\begin{proof}
Suppose for contradiction that $\PP$ has a triangular face, $\Delta_1$.  Andreev's
theorem for $\pi/3$--equiangular polyhedra (Theorem~\ref{and3}) says that the dual graph of $\PP^{(1)}$ has no
prismatic $3$--circuits, so at least two of the edges emanating from $\Delta_1$
share a vertex.  Hence $\PP$ contains two adjacent triangular faces, $\Delta_1$
and $\Delta_2$.  Let $v_1$ be the vertex of $\Delta_1$ which is not contained in
$\Delta_2$ and $v_2$ the vertex of $\Delta_2$ which is not contained in
$\Delta_1$.
Let $e_1$ and $e_2$ be the edges emanating from $v_1$ and $v_2$ respectively which are
not contained in $\Delta_1$ or $\Delta_2$.   The edges $e_1$ and $e_2$ are both
contained in two common faces.  See figure~\ref{figure:andedge}.  Therefore by
convexity, $e_1$ and $e_2$ must actually be the same edge, which contradicts the
fact that $\PP$ has more than $4$ vertices.
\end{proof}

\begin{figure} 
\labellist
\small\hair 2pt
\pinlabel $e_1$ [t] at 249 448 
\pinlabel $e_2$ [t] at 305 448 
\pinlabel $v_1$ [tr] at 236 450
\pinlabel $v_2$ [tl] at 319 450
\pinlabel $\Delta_1$ [tr] at 258 492
\pinlabel $\Delta_2$ [tl] at 300 492
\endlabellist
\begin{center}
\scalebox{1}{\includegraphics{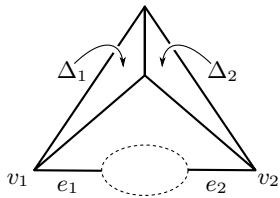}} 
\end{center}
\caption{Note that $e_1$ and $e_2$ are both part of the ``front" face and the ``back"
face.} 
\label{figure:andedge} 
\end{figure}

\begin{lemma}\label{cornersymmetry}
Suppose that $\PP$ is an ideal $\pi/3$--equiangular polyhedron with $N>4$
vertices.  If coordinates for the upper half-space model of $\HH^3$ are chosen
so that a vertex of $\PP$ is at the point at infinity, then the Euclidean
distance from a corner vertex to each of the adjacent vertices in the base
triangle are equal.
\end{lemma}

\begin{proof}
The fact that all dihedral angles are equal to $\pi/3$ implies that the
arrangement of defining planes for the corner vertex is left invariant under reflection
through a geodesic plane through infinity bisecting the angle between the two
vertical planes defining the vertex.  This proves the lemma.
\end{proof}

The following lemma shows that if $\PP$ is a non-obtuse polyhedron, then
intersections of faces of $\PP$ correspond to intersections of the defining
planes of $\PP$.  

\begin{lemma}\label{orthgeod}
If $\PP$ is a non-obtuse hyperbolic polyhedron, then the closures of two faces
$F_1$ and $F_2$ of $\PP$ intersect if and only if $\overline{\Pi}_1$ and 
$\overline{\Pi}_2$ intersect in $\overline{\HH}^3$ where  $\Pi_i$ is the
defining plane for $F_i$.
\end{lemma}

\begin{proof}
Sufficiency is clear.

For necessity, the contrapositive will be proved.  Suppose that $F_1$ and $F_2$
are two faces of $\PP$ such that their closures do not intersect.  A geodesic
orthogonal to both $F_1$ and $F_2$ will be constructed.  This geodesic is also
orthogonal to both $\Pi_1$ and $\Pi_2$, and such a geodesic exists only if the
closures of the $\Pi_i$ are disjoint.
 
Choose any $x_0 \in F_1$ and $y_0 \in F_2$ and let $\gamma_0$ be the geodesic
between them.  The set, $$K(\gamma_0) = \{(x,y) \in F_1 \times F_2 \, \mid \,
d(x,y) \leq l(\gamma_0) \},$$ is a closed subset of $F_1 \times F_2$.  
There exists open subsets $N_i$ of $F_i$ containing all of the ideal vertices
of $F_i$ such that for any $z_1 \in N_1$ and $z_2 \in N_2$, $d(z_1,z_2)
> l(\gamma_0).$  Hence $K(\gamma_0)$ is also a bounded subset of $F_1 \times
F_2$, therefore compact.

It follows from compactness of $K(\gamma_0)$ that $$d_{min} = \min \{ d(x,y)\,
\mid \, (x,y) \in K(\gamma_0) \}$$ is achieved for some $(x,y) \in K(\gamma_0).$
The geodesic segment, $\gamma$, between $x$ and $y$ must be orthogonal to both
$F_1$ and $F_2$.  If not, suppose $\gamma$
is not orthogonal to $F_1$.  Since $\PP$ is non-obtuse, the orthogonal
projection of $y$ to $\Pi_1$ is contained in $F_1$.  By the hyperbolic
Pythagorean theorem, the geodesic between $y$ and its projection has
length less than that of $\gamma$.  This contradicts the construction of
$\gamma$.  The argument is identical if $\gamma$ is not orthogonal to $F_2$. 
\end{proof}

There is actually a simpler argument for the previous lemma in the case that
$\PP$ is a Coxeter polyhedron.  The development of $\PP$ into $\HH^3$ gives a
tessellation of $\HH^3$ by copies of $\PP$, so any intersection of defining planes must
correspond to an edge of $\PP$.

\begin{corollary}\label{3lem} 
Suppose $\PP$ is an ideal $\pi/3$--equiangular polyhedron with $N>4$
vertices.   Choose coordinates for the upper half-space model of $\HH^3$ so that
a vertex, $v_0,$ of $\PP$ is at the point at infinity.  Then if the distance in
Lemma \ref{cornersymmetry} from the corner vertex, $u$, to the two adjacent
vertices in the base triangle is $r$ and the edge length of the base triangle is
$a$, then $0<r<\frac{3a}{4}$.  
\end{corollary}

\begin{proof} 
Suppose that the three defining planes which contain $v_0$ are $\Pi_1$, $\Pi_2$
and $\Pi_3$ and that the three defining planes containing $u$ are $\Pi_1$,
$\Pi_2$ and $\Pi_4$.   If $\frac{3a}{4} \leq r < a,$ then $\Pi_3$ intersects
$\Pi_4$ with interior dihedral angle less than $\pi/3$.   By
Lemma~\ref{orthgeod}, the corresponding faces, $F_3$ and $F_4$, also have
intersecting closures, and the interior dihedral angle between $F_3$ and $F_4$
will be less than $\pi/3$.  If $r=a$, then $\PP$ would be a tetrahedron and for
$r > a$, $\PP$ would have finite vertices at the points $\Pi_1 \cap \Pi_3 \cap
\Pi_4$ and $\Pi_2 \cap \Pi_3 \cap \Pi_4$.
\end{proof}

In what follows, the intersection with $\PP$ of a closed horoball centered at a
vertex $u$ of $\PP$ which intersects only faces and edges containing $u$ will be
called a {\it vertex neighborhood}.  The next lemma is the main
observation which leads to the lower volume bound.  This lemma follows the
approach of Adams in \cite{adams}.
 
\begin{lemma}\label{bump}
Let $\PP$ be an ideal $\pi/3$--equiangular polyhedron with more than $4$
vertices.  Suppose two vertex neighborhoods of equal volume intersect with
disjoint
interiors.  Then the volume of each of
the vertex neighborhoods is at least $\frac{\sqrt{3}}{6}$.
\end{lemma}
\begin{proof}  Let $B_1$ and $B_2$ about vertices $u_1$
and $u_2$, respectively, be vertex neighborhoods which intersect with disjoint
interiors.  Choose coordinates for the upper half-space model of $\HH^3$ so that
$u_1$ is the point at infinity and $B_1$ intersects $B_2$ at Euclidean height
$1$ above the bounding plane.  Let $\Gamma(\PP)$ be the reflection group
generated by $\PP$ and $\GG_{\infty}$ the subgroup fixing the point at
infinity: 
$$\GG_{\infty} = \{ \gamma \in \Gamma(\PP)\, \mid \,
\gamma(\{\infty\})=\{\infty\}\}.$$
  Let $H_1=\GG_{\infty} \cdot B_1$ be the horoball centered
at infinity covering $B_1$ and let $H_2$ be one of the height $1$ horoballs
contained in $\Gamma(\PP) \cdot B_2$.  The projection of $\PP$ to the bounding
plane of $\HH^3$ is an equilateral triangle and the orbit of this triangle under
the action of $\GG_{\infty}$ tiles the plane.  Let $\DD$ be a triangle in
this tiling containing the point of $\overline{\HH^3}$ about which $H_2$ is
centered. 

The collection of height $1$ horoballs covering $B_2$ is equal to $\GG_{\infty}
\cdot H_2$ and, for each pair $g\neq h \in \GG_{\infty},$ either $g H_2 \cap h H_2$ is empty, a
single point, or $g H_2 = h H_2$.  The proof breaks up into three cases.  Either
$u_2 \in \text{int}\,{\DD}$, $u_2$ is contained in the interior of an
edge of $\DD$, or $u_2$ is a vertex of $\DD$.

If $u_2 \in \text{int}\,{\DD},$ then the projection of $H_2$ to the bounding
plane must be a closed disk contained in $\DD$.  The minimum possible value of
$\vol(B_1)=\vol(B_2)$ occurs when the projection of $H_2$ to the bounding plane
is inscribed in $\DD$ and $\DD$ has edge length $\sqrt{3}$, as shown in
figure~\ref{baryball}.  Hence the area of $\DD$ is $\frac{3 \sqrt{3}}{4}$ and 
$$\vol(B_i)=\frac{3 \sqrt{3}}{4}\int_{1}^{\infty} \frac{dz}{z^3} =
\frac{3\sqrt{3}}{8}.$$  

\begin{figure}
\begin{center}
\scalebox{.65}{\includegraphics{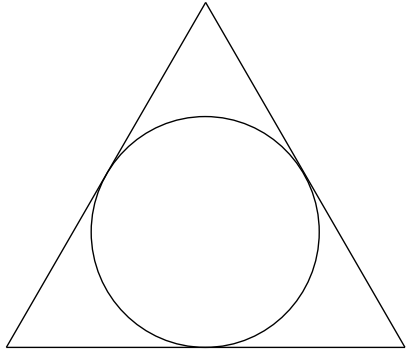}}
\end{center}
\caption{Projection of $H_2$ inscribed in $\DD$ \label{baryball}}
\end{figure}

If $u_2$ is contained in the interior of an edge of $\DD$, then the minimum
possible value of the vertex neighborhood volume occurs when $\DD$ has
edge length $\frac{2\sqrt{3}}{3}$ and $u_2$ is at the midpoint of an edge of
$\DD$, as in figure~\ref{edgeball}.  Calculating as above, 
$$\vol(B_i)=\frac{\sqrt{3}}{6}.$$

\begin{figure}
\begin{center}
\scalebox{.55}{\includegraphics{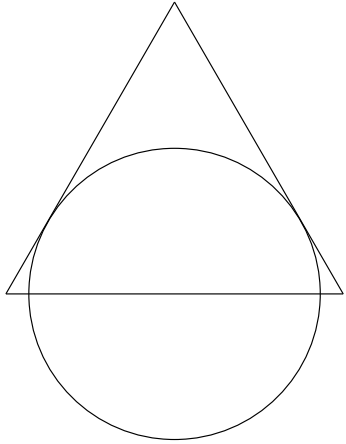}}
\end{center}
\caption{Projection of $H_2$ for case $u_2$ contained in the interior of an
edge of $\DD$ \label{edgeball}}
\end{figure}

Now suppose that $u_2$ is a vertex of $\DD$.  If the edge length of $\DD$ is
$a$, then 
$$\vol(B_1)=\frac{a^2 \sqrt{3}}{4}\int_{1}^{\infty} \frac{dz}{z^3}=\frac{a^2
\sqrt{3}}{8}.$$
If $r$ is as in Corollary~\ref{3lem}, then 
$$\vol(B_2)=\frac{\sqrt{3}}{8 r^2}.$$
By Corollary~\ref{3lem}, $0<r<\frac{3a}{4}$, so 
$$\vol(B_2)> \frac{2 \sqrt{3}}{9a^2}.$$
Equating this lower bound with $\vol(B_1)$ yields $a= \frac{2\sqrt{3}}{3}$.
Therefore it may concluded that
$$\vol(B_i)>\frac{\sqrt{3}}{6}.$$
\end{proof}

To complete the proof of Theorem~\ref{pi3prop}, start with disjoint, equal
volume vertex neighborhoods at each vertex.  Expand the vertex
neighborhoods so that the volumes remain equal at all time until two of the
vertex neighborhoods intersect with disjoint interior intersection.
Lemma~\ref{bump} then says that there is a vertex neighborhood at each vertex of volume
at least $\sqrt{3}/6$.  B\"or\"oczky and Florian in \cite{boro} show that the
maximal density of a horoball packing in $\HH^3$ is ${\sqrt{3}}/(2V_3)$.  Applying
this result gives 
$$\vol(\PP) > N\cdot \frac{V_3}{3}.$$  

In a personal communication, Rivin has indicated how to improve the lower bound
to $N\cdot \frac{3V_3}{8}.$  The idea of the argument is that for any given
vertex $v$ in a $\pi/3$-equiangular polyhedron $\PP$, $v$ along with the three
vertices of $\PP$ with which $v$ shares an edge are the vertices of a regular
ideal hyperbolic tetrahedron contained in $\PP$.  A collection of such
tetrahedra with disjoint interiors may be constructed by taking any independent
set of vertices of $\PP$.  By a result of Heckman and Thomas, a trivalent graph
with $N$ vertices contains an independent set of cardinality at least
$\frac{3N}{8}$ \cite{indset}.

\section{The Upper Volume Bounds}\label{S:upper}

In this section, the upper volume bounds in Theorems~\ref{theorempi2},
\ref{pi2finite}, \ref{pi2general}, and \ref{theorempi3} will be proved using 
 arguments inspired by an argument of Agol and D.  Thurston for an upper bound
on the volume of an alternating link complement \cite{lackenby}.  First, a
decomposition of an arbitrary non-obtuse hyperbolic polyhedron into tetrahedra
will be described.  In each case, the volume contributed by the tetrahedra
meeting at each vertex will be analyzed to obtain the volume bounds. 

Let $\PP$ be a non-obtuse hyperbolic polyhedron and $v_0$ a vertex of $\PP$.
For each face, $A_i$, not containing $v_0$, let $\gamma_i$ be the unique geodesic
orthogonal to $A_i$ which passes through or limits to $v_0$, where $v_0$ is a
finite or ideal vertex respectively.  Define the nearest point projection,
$u_i$, of $v_0$ to $A_i$ to be the intersection of $\gamma_i$ with $A_i$.  The
projection $u_i$ will lie on the interior of $A_i$ unless $A_i$ meets one of the
faces containing $v_0$ orthogonally, in which case, $u_i$ will lie in the
interior of an edge of $A_i$ or will coincide with a vertex of $A_i$ if $A_i$
meets two faces containing $v_0$ orthogonally.
Cyclically label the vertices of $A_i$ by $v_{i,j}$ where  $j\in \{ 1,\, 2, \dots
\deg(A_i)\}$ is taken modulo $\deg(A_i)$.    Let $w_{i,j}$ be the nearest point
projection of $u_i$ onto the edge of $A_i$ with endpoints $v_{i,j}$ and
$v_{i,j+1}$, where the nearest point projection is defined as above.  Each face
of a non-obtuse polyhedron is a non-obtuse polygon, so the nearest point
projection of any point in $A_i$ to an edge $A_i$ actually lies in $A_i$.  See
figure~\ref{figure:tri}.

\begin{figure} 
\labellist
\small\hair 2pt
\pinlabel $u_i$ at 103 502
\pinlabel $v_{i,1}$ [b] at 106 563 
\pinlabel $v_{i,2}$ [bl] at 175 543 
\pinlabel $v_{i,3}$ [l] at 193 496
\pinlabel $v_{i,4}$ [t] at 154 432
\pinlabel $v_{i,5}$ [tr] at 70 450
\pinlabel $v_{i,6}$ [r] at 53 497
\pinlabel $w_{i,1}$ [b] at 143 554
\pinlabel $w_{i,2}$ [bl] at 184 518
\pinlabel $w_{i,3}$ [tl] at 174 464
\pinlabel $w_{i,4}$ [t] at 111 440
\pinlabel $w_{i,5}$ [r] at 62 474
\pinlabel $w_{i,6}$ [br] at 79 529
\pinlabel $u_i$ [r] at 266 499
\pinlabel $v_{i,1}$ [br] at 268 547 
\pinlabel $v_{i,2}$ [b] at 329 570
\pinlabel $v_{i,3}$ [l] at 390 510
\pinlabel $v_{i,4}$ [t] at 336 433
\pinlabel $v_{i,5}$ [tr] at 266 458
\pinlabel $w_{i,1}$ [b] at 285 555
\pinlabel $w_{i,2}$ [bl] at 345 556
\pinlabel $w_{i,3}$ [tl] at 343 444
\pinlabel $w_{i,4}$ [tr] at 288 449
\endlabellist
\begin{center}
\scalebox{1}{\includegraphics{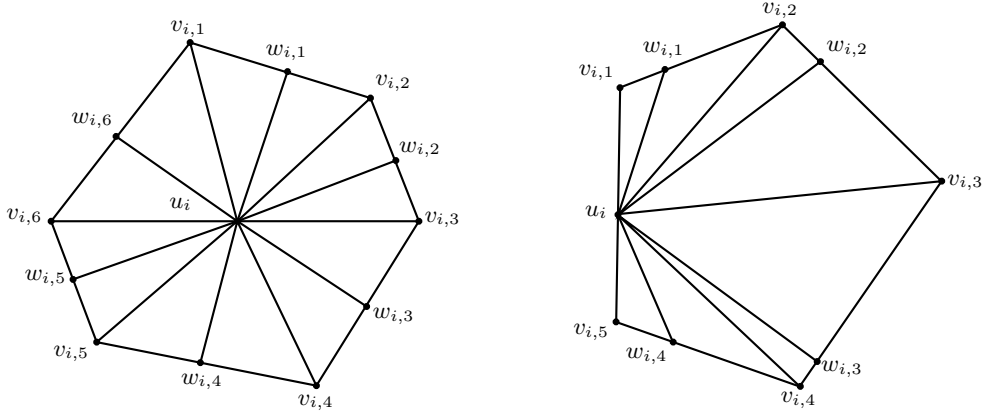}} 
\end{center}
\caption{The figure on the left shows the case where $u_i$ is in the interior of
a face. The figure on the right is the case where $u_i$ is in the interior of
an edge.} 
\label{figure:tri} 
\end{figure}

Define $\Delta(i,j)$ to be the tetrahedron with vertices $v_0$, $u_i$,
$w_{i,j}$, and $v_{i,j}$ and $\Delta'(i,j)$ to be the tetrahedron with vertices $v_0$, $u_i$,
$w_{i,j}$, and $v_{i,j+1}$.  In the case where $u_i$ coincides with $w_{i,j}$,
both $\Delta(i,j)$ and $\Delta'(i,j)$ will be degenerate tetrahedra.   Let
$\mathcal{I}$ be the set of $(i,j)$ such that $\Delta(i,j)$ and $\Delta'(i,j)$
are nondegenerate.   For each $(i,j)\neq (i',j') \in \mathcal{I}$,
$\text{Int}(\Delta(i,j))\cap \text{Int}(\Delta(i',j'))=\emptyset$ and $\text{Int}(\Delta'(i,j))\cap
\text{Int}(\Delta'(i',j'))=\emptyset.$  Also, the interior of each $\Delta$ is
disjoint from the interior of each $\Delta'$.   Then
$$\PP= \bigcup_{(i,j) \in \mathcal{I}} \left( \Delta(i,j) \cup \Delta'(i,j)
\right).$$

This decomposition of $\PP$ into tetrahedra will be analyzed to prove each of the upper bounds in Theorems~\ref{theorempi2},
\ref{pi2finite}, \ref{pi2general}, and \ref{theorempi3}.  The following
technical lemma is needed.  It follows directly from the fact that the
Lobachevsky function is concave down on the interval $[0,\pi/2]$. 

\begin{lemma}\label{maxlem}
Suppose $\overrightarrow{\alpha}=(\alpha_1, \dots, \alpha_M)$ where
$\alpha_i\in[0,\pi/2]$.  Let  
$$f(\overrightarrow{\alpha})=\frac{1}{2}\sum_{i=1}^{M} \Lambda(\pi/2-\alpha_i)$$ 
and $g(\overrightarrow{\alpha})=\alpha_1+ \cdots + \alpha_M$.  Then the maximum
value of $f(\overrightarrow{\alpha})$ subject to the constraint
$g(\overrightarrow{\alpha})=C$ for some constant $C \in [0, M\pi/2]$ occurs for
$\overrightarrow{\alpha}=(C/M, \dots, C/M)$.
\end{lemma}

The next proposition gives the upper bound in Theorem~\ref{theorempi2}.
\begin{proposition}\label{upper2}
If $\PP$ is an ideal $\pi/2$--equiangular polyhedron with $N$ vertices, then
$$\vol(\PP) \leq (N-4) \cdot \frac{V_8}{2},$$
where $V_8$ is the volume of the regular ideal hyperbolic octahedron.
Equality is achieved when $\PP$ is the regular ideal hyperbolic octahedron.
\end{proposition} 

\begin{proof}
Decompose $\PP$ as above.  Suppose that $v=v_{i,j}$ is a vertex of $\PP$ which
is not contained in a face containing $v_0$.  Then $v$ is contained in exactly
eight tetrahedra of the decomposition, say
$T_1, \dots, T_8$. Suppose that $T_l$ coincides with $\Delta(m,n)$ in the
decomposition.  Then $T_l$ is a tetrahedron with $2$ ideal vertices, $v_0$ and
$v$, and two finite vertices, $u_m$ and $w_{m,n}$.  The dihedral angles along
the edges between $v$ and $u_m$, between $u_m$ and $w_{m,n}$, and between
$w_{m,n}$ and $v_0$ are all $\pi/2$.  Suppose that the dihedral
angle along the edge between $v$ and $v_0$ is $\alpha_l$.  Then the dihedral
angles along the remaining two edges are $\pi/2-\alpha_l$.  See
figure~\ref{figure:tet}.   

\begin{figure} 
\labellist
\small\hair 2pt
\pinlabel $u_m$ [l] at 192 170
\pinlabel $w_{m,n}$ [tl] at 169 126
\pinlabel $\alpha_l$ [l] at 84 177
\pinlabel $\frac{\pi}{2}$ [l] at 168 177
\pinlabel $\frac{\pi}{2}-\alpha_l$ [l] at 192 190
\pinlabel $v$ [tr] at 84 46
\endlabellist
\begin{center}
\scalebox{1}{\includegraphics{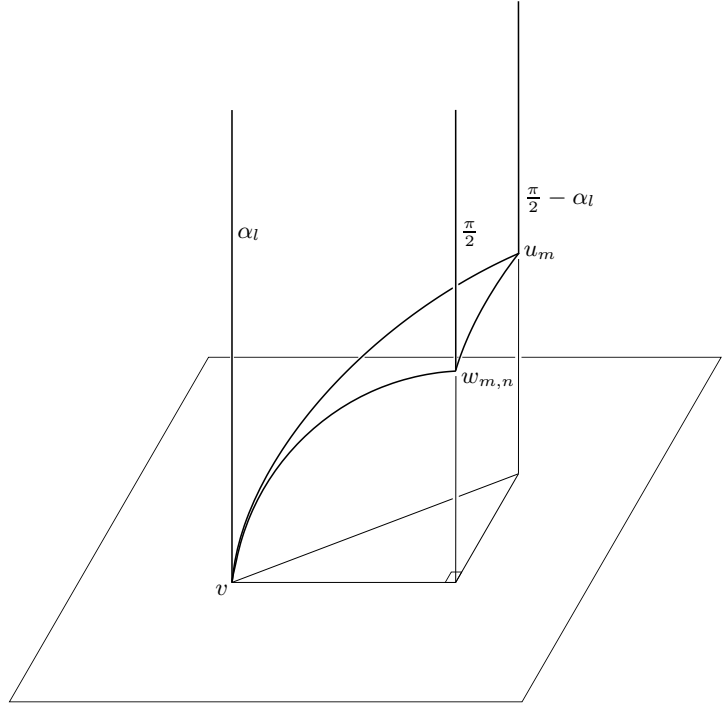}} 
\end{center}
\caption{One of the $T_l$} 
\label{figure:tet} 
\end{figure}

The volume of $T_l$ is given by $\Lambda(\pi/2-\alpha_l)/2$, where the
Lobachevsky function, $\Lambda,$ is defined as 
$$\Lambda(\theta)=-\int_0
^{\theta} \log |2 \sin(t) | \, dt.$$ 
Therefore, the
volume contributed by the tetrahedra adjacent to the vertex $v$ is a function of
$\overrightarrow{\alpha}=(\alpha_1, \alpha_2, \dots, \alpha_8)$:
$$f(\overrightarrow{\alpha})= \frac{1}{2} \sum_{l=1}^{8}
\Lambda(\pi/2-\alpha_l).$$

 The $\alpha_l$ must sum to $2\pi$.  The maximum of
$f(\overrightarrow{\alpha})$, subject to this constraint, occurs when
$\overrightarrow{\alpha}=(\pi/4, \pi/4, \dots, \pi/4)$, by Lemma~\ref{maxlem}.  Gluing $16$ copies of
$T_l$ with $\alpha_l=\pi/4$ together appropriately yields  a regular ideal
hyperbolic octahedron.  Hence $f(\pi/4, \pi/4, \dots, \pi/4) = V_{8}/2.$

Only two tetrahedra in the decomposition meet each of the four vertices which
share an edge with $v_0$.  A similar analysis as above shows that the
volume contributed by the tetrahedra at each of these four vertices is no more
than $V_{8}/8$.  

Therefore, accounting for the vertex, $v_0$, at infinity and the fact that only
$V_8 /8$ is contributed by each of the tetrahedra at the vertices adjacent to
$v_0$,
$$\vol(\PP) \leq (N-1)\cdot \frac{V_8}{2} - 4 \cdot
3\frac{V_8}{8}=(N-4) \cdot
\frac{V_8}{2}.$$
Equality is clearly achieved when $\PP$ is the regular ideal hyperbolic
octahedron.
\end{proof}

The proof of the upper bound in Theorem~\ref{theorempi3} is similar to the
previous argument.  

\begin{proposition}\label{upper3}
If $\PP$ is an ideal $\pi/3$--equiangular polyhedron with $N$ vertices, then
$$\vol(\PP) \leq (3N-14) \cdot \frac{V_3}{2},$$
where $V_3$ is the volume of the regular ideal hyperbolic tetrahedron.  Equality
is achieved when $\PP$ is the regular ideal hyperbolic cube.
\end{proposition}

\begin{proof}
Decompose $\PP$ as described at the beginning of this section.
Each vertex of $\PP$ which is not contained in a face containing $v_0$ is a
vertex of exactly six tetrahedra of the decomposition.  Lemma~\ref{maxlem}
implies that the sum of the volumes of the six tetrahedra around such a vertex
is no more than $3V_3/2$.

If $v_1$ is one of the three vertices adjacent to $v_0$, then $v$ is a vertex of
two tetrahedra, $T_1$ and $T_2$, say.  The sum of the volumes of $T_1$ and $T_2$
is at most $V_3/3$ when $\alpha_1=\alpha_2=\pi/6$, again by Lemma~\ref{maxlem}.

By Corollary~\ref{cor4edges} each face containing $v_0$ has degree at least $4$.
If $v_2$ is a vertex of such a face which does not share an edge with $v_0$,
then $v_2$ is a vertex of four tetrahedra of the decomposition of $\PP$. See
figure~\ref{figure:edgevert}.

\begin{figure} 
\labellist
\small\hair 2pt
\pinlabel $v_2$ [br] at 331 289
\pinlabel $T_3$ at 358 303
\pinlabel $T_4$ at 330 270
\pinlabel $T_5$ [t] at 360 246
\pinlabel $T_6$ at 363 277
\endlabellist
\begin{center}
\scalebox{.7}{\includegraphics{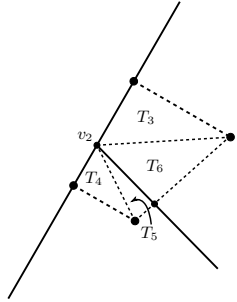}} 
\end{center}
\caption{A view of $v_2$ as seen from $v_0$.  The solid lines are edges of $\PP$, and the dashed lines are edges of
tetrahedra which are not also edges of $\PP$.} 
\label{figure:edgevert} 
\end{figure}

The link of $v_2$ intersected with each of $T_i$, $i = 3,4$  is a Euclidean
triangle with angles $\pi/2,$ $\pi/3$ and $\alpha_i$.  Hence $\alpha_3 =
\alpha_4 = \pi/6$.  Using
Lemma~\ref{maxlem} and the fact that
$\alpha_5+\alpha_6=2\pi/3$, the sum of the volumes of these four tetrahedra is
seen to have a maximum value of $5V_3/6$.

The upper bound is computed by assuming that the volume contributed by the
tetrahedra containing each vertex other than $v_0$ is $3V_3/2$ and subtracting
the excess for each of the three vertices which share an edge with $v_0$ and for
the three vertices described in the previous paragraph:
$$\Vol(\PP) \leq \left( (N-1)\frac{3}{2} - 3\left(\frac{7}{6}
\right) -
3\left(\frac{2}{3}\right)\right) \cdot
V_3=\left(\frac{3N-14}{3}\right)\cdot \frac{3V_3}{2}.$$

The regular ideal hyperbolic cube has $N=8$ and volume $5V_3$.
\end{proof}

The proofs of Theorems~\ref{pi2finite} and \ref{pi2general} require different
methods than the previous two theorems because Lemma~\ref{maxlem} does not
apply.  The volume of the tetrahedra into which $\PP$ is decomposed is given
by the sum of three Lobachevsky functions, so the simple Lagrange multiplier
analysis fails.  The next lemma will play the role of Lemma~\ref{maxlem} in what
follows.

\begin{lemma}\label{cube}
The regular ideal hyperbolic cube has largest volume among
all ideal polyhedra with the same combinatorial type.
\end{lemma}

\begin{proof}
Any ideal polyhedron, $\QQ$, with the combinatorial type of the cube can be
decomposed into five ideal tetrahedra as follows:  Let $v_1, v_2, v_3, v_4$ be a
collection of vertices of $\QQ$ so that no two share an edge.  The five ideal
tetrahedra consist of the tetrahedron with vertices $v_1, v_2, v_3,$ and $v_4$,
and the four tetrahedra with vertices consisting of $v_i$ along with the three
adjacent vertices, for $i=1,2,3,4$.

Then since the regular ideal tetrahedron is the ideal tetrahedron of maximal
volume, $\vol(\QQ)\leq 5V_3$.  The regular ideal hyperbolic cube is decomposed
into five copies of the regular ideal tetrahedron when the above decomposition
is applied, so has volume $5V_3$.
\end{proof}

The next proposition proves the upper bound in Theorem~\ref{pi2finite}.
\begin{proposition}\label{upper2finite}
If $\PP$ is a $\pi/2$--equiangular compact hyperbolic polyhedron with $N$
vertices, then 
$$\vol(\PP) < (N-10)\cdot \frac{5V_3}{8}.$$
\end{proposition}

\begin{proof}
Decompose $\PP$ into tetrahedra as described
at the beginning of this section for some choice of $v_0$.  The volume
of $\PP$ will be bounded above by considering tetrahedra with one ideal vertex.
The reason for using tetrahedra with an ideal vertex to estimate the volume of a
compact polyhedron is that 
$$\max_{v\in \text{Vert}(\PP)} d(v_0,v)$$
can be made arbitrarily large by choosing polyhedra $\PP$ with a large enough
number of vertices.

Suppose that $v=v_{i,j}$ is a vertex of $\PP$ which is not contained in a face
containing $v_0$.  The vertex $v$ is contained in six tetrahedra of the
decomposition.  Consider $S$, the union of the six triangular faces of these
tetrahedra which are contained in the faces of the polyhedron which contain
 $v$.  Let $\hat{v}$ be the point at infinity determined by the geodesic ray
emanating from $v$ and passing through $v_0$.  Define $\TT$ to be the cone of
$S$ to $\hat{v}$.  

The cone, $\TT,$ is an octant of an ideal cube, $\QQ$.  By Lemma~\ref{cube},
$\vol(\QQ) \leq 5V_3.$ Then since $\vol(\QQ) = 8 \vol(\TT),$  $\vol(\TT) \leq
\frac{5V_3}{8}.$

By Andreev's theorem, each face of $\PP$ must be of degree at least $5$.  Hence,
 the three faces of $\PP$ containing $v_0$ contain at least $10$ distinct vertices of
$\PP$, so there are at most $N-10$ vertices that do not share a face with $v_0$.
Therefore the volume of $\PP$ satisfies
$$\vol(\PP) < (N-10)\cdot \frac{5V_3}{8}.$$
\end{proof}

Proposition~\ref{upper2general} combines the techniques of
Proposition~\ref{upper2} and Proposition~\ref{upper2finite} and gives the upper
bound in Theorem~\ref{pi2general}.

\begin{proposition}\label{upper2general}
If $\PP$ is a $\pi/2$--equiangular hyperbolic polyhedron with $N_{\infty}\geq 1$ ideal
vertices and $N_F$ finite vertices, then 
$$\vol(\PP) < (\N_{\infty}-1)\cdot \frac{V_8}{2} +
 N_F\cdot \frac{5 V_3}{8} .$$
\end{proposition}

\begin{proof}
Assign to one of the ideal vertices the role of $v_0$ in the decomposition
described at the beginning of this section.  Then each ideal vertex of $\PP$
which is not contained in a face of $\PP$ containing $v_0$ will be a vertex of
eight tetrahedra in the decomposition.  These tetrahedra contribute no more than $V_8/2$ to
the volume of $\PP$, by Proposition~\ref{upper2}.  Each finite vertex which is not contained in a face
containing $v_0$ is a vertex of six tetrahedra.  The volume contributed by
these is no more than $5V_3/8$ by Lemma~\ref{cube} and the proof of
Proposition~\ref{upper2finite}.  Putting this all
together completes the proof. 
\end{proof}


\section{Sequences of polyhedra realizing the upper bound estimates}
\label{S:seq}

In this section, it is proved that the upper bounds in
Theorems~\ref{theorempi2}, in \ref{pi2finite}, and
in \ref{theorempi3} are asymptotically sharp.  Results will first be established
about the convergence of sequences of circle patterns in the plane and about the
convergence of volumes of polyhedra which correspond to these circle patterns. 

Define a {\it disk pattern} to be a collection of closed round disks in the
plane such that no disk is the Hausdorff limit of a sequence of distinct disks
and so that the boundary
of any disk is not contained in the union of two other disks.  Define the angle between two disks to be the
angle between a clockwise tangent vector to the boundary of one disk at an
intersection point of their boundaries and a counterclockwise tangent vector to
the boundary of the other disk at the same point.  Suppose that $D$ is a disk
pattern such that for any two intersecting disks, the angle between them is in the
interval $[0,\pi/2]$.  Define $G(D)$ to be the graph with a vertex for each disk
and an edge between any two vertices whose corresponding disks have non-empty
interior intersection.  The graph $G(D)$ inherits an embedding in the plane from
the disk pattern.  Identify $G(D)$ with its embedding.  A face of $G(D)$ is a
component of the complement of $G(D)$.  Label
the edges of $G(D)$ with the angles between the intersecting disks.  The graph
$G(D)$ along with its edge labels will be referred to as the {\it labeled
$1$--skeleton for the disk pattern $D$}.  A disk pattern $D$ is said to be {\it
rigid} if $G(D)$ has only triangular and quadrilateral faces and each
quadrilateral face has the property that the four corresponding disks of the
disk pattern intersect in exactly one point.  See \cite{he} for more details on
disk patterns.

Consider the path metric on $G(D)$ obtained by giving each edge of $G(D)$ length
$1$.  Given a disk $d$ in a disk pattern $D$, the set of disks corresponding to
the ball of radius $n$ in $G(D)$ centered at the vertex corresponding to $d$ will
be referred to as {\it $n$ generations of the pattern about $d$}.  Given disk
patterns $D$ and $D'$ and disks $d\in D$ and $d' \in D'$, then $(D,d)$ and
$(D',d')$ {\it agree to generation $n$} if there is a label preserving graph
isomorphism between the balls of radius $n$ centered at the vertices
corresponding to $d$ and $d'$.  The following proposition is  a slight
generalization of the Hexagonal Packing Lemma in \cite{rodinsullivan}.

\begin{proposition}\label{hrs}
Let $c_{\infty}$ be a disk in an infinite rigid disk pattern
$D_{\infty}$.  For each positive integer $n$, let $D_n$ be a rigid finite disk
pattern containing a disk $c_n$ so that $(D_{\infty},c_{\infty})$ and
$(D_n,c_n)$ agree to generation $n$.  Then there exists a sequence $s_n$
decreasing to $0$ such that the ratios of the radii of any two disks adjacent to
$c_n$ differ from $1$ by less than $s_n$.  
\end{proposition}       

\begin{proof} 
With lemma 7.1 from \cite{he} playing the role of the ring lemma
in \cite{rodinsullivan}, the proof runs exactly the same.  The length--area
lemma generalizes to this case with no change and any reference to the
uniqueness of the hexagonal packing in the plane should be replaced with
Rigidity Theorem 1.1 from \cite{he}.  
\end{proof}

A {\it simply connected disk
pattern} is a disk pattern so that the union of the disks is simply connected.
Disk patterns arising from finite volume hyperbolic polyhedra will all be simply
connected, so all disk patterns will be implicitly assumed to be simply
connected.  If for a simply connected disk pattern $D$, all labels on $G(D)$ are
in the interval $(0,\pi/2]$, Andreev's theorem implies that each face of $G(D)$
will be a triangle or quadrilateral.  An {\it ideal disk pattern}, $D$, is one
where the labels of $G(D)$ are in the interval $(0,\pi/2]$ and the labels around
each triangle or quadrilateral in $G(D)$ sum to $\pi$ or $2\pi$ respectively.
Ideal disk patterns correspond to ideal polyhedra.  A {non-ideal disk pattern},
$D$, is one where $G(D)$ has only triangular faces and the sum of the labels
around each
face is greater than $\pi$.  These disk patterns correspond to compact
polyhedra.  
 
Ideal disk patterns and their associated polyhedra will be dealt with first.
The analysis for non-ideal disk patterns is slightly different and will be
deferred until after the proofs of the remaining claims in
Theorems~\ref{theorempi2} and \ref{theorempi3}.  The upper half-space model for
$\HH^3$ will be used here.  For each disk $d$ in the
circle pattern, let $S(d)$ be the geodesic hyperbolic plane in $\HH^3$ bounded
by the boundary of $d$.  

Suppose $c$ is a disk in $D$ which intersects $l$ neighboring disks, $d_{1},
\dots d_{l}$.  In the case of an ideal disk pattern, the intersection of $S(c)$
with each of the $S(d_{i})$ is a hyperbolic geodesic.  These $l$ geodesics bound
an ideal polygon, $p(c) \subset \HH^3$. If necessary, choose coordinates so that
the point at infinity is not contained in $c$.  Cone $p(c)$
to the point at infinity and denote the ideal polyhedron thus obtained by
$C(p(c))$.  

\begin{lemma}\label{vc}
Suppose that $D_n$ and $D_{\infty}$ are simply connected, ideal, rigid, disk patterns
such that $(D_n,c_n)$ and $(D_{\infty},c_{\infty})$ satisfy
Proposition~\ref{hrs}.
Then
  $$\lim_{n\to \infty}
\vol(C(p(c_n)))=\vol(C(p(c_{\infty})))$$ 
where $C(p(c_{\infty}))$ is the cone on
the polygon determined by the disk $c_{\infty}$.  Moreover, there exists a bounded
sequence $0\leq \epsilon_n \leq K <\infty$ converging to zero such that  $|
\vol(C(p(c_n)))-\vol(C(p(c_{\infty}))) | \leq \epsilon_n$. 
\end{lemma}
\begin{proof} Suppose the dihedral angle between $S(c_n)$ and the vertical face which is
a cone on the intersection of $S(c_n)$ and $S(d_{i,n})$ is $\aa_{i}^{n}$ and
that the corresponding dihedral angles in $C(p(c_{\infty}))$ are
$\aa_i^{\infty}$.  Then by chapter 7 of \cite{thurstonnotes},
$$\vol(C(p(c_n)))=\sum_{i=1}^{l} \Lambda(\aa_i^{n}),$$ 
where $p(c_n)$ has degree $l$.  For each $i$, $\aa_i^{n}$
converges to $\aa_i ^{\infty}$ because $\aa _i^{n}$ is a continuous function of
the angle between $c_n$ and $d_{i,n}$ and the radii of the two disks, which
converge to the radii of the corresponding disks in the infinite packing by
Proposition~\ref{hrs}.  The function $\Lambda$ is continuous, so convergence of
the $\aa$'s implies the first statement of the lemma.  The second statement is a
consequence of the first and the fact that $\vol(C(p(c_n)))$ is finite for all
$n$ including $\infty.$  
\end{proof}

The remaining claims in Theorems \ref{theorempi2} and \ref{theorempi3} can now
be proved.  First the following proposition is proved:

\begin{proposition}\label{asy2}
There exists a sequence of ideal $\pi/2$--equiangular polyhedra $\PP_i$ with $N_i$ vertices such that
$$\lim_{i\to \infty} \frac {\vol(\PP_i)}{N_i} = \frac{V_8}{2}.$$
\end{proposition}

\begin{proof}
Let $D_{\infty}$ be the infinite disk pattern defined as 
$$D_{\infty}=\bigcup d_{(p,q)},$$
where the union ranges over all $(p,q)\in \ZZ^2$ such that both $p$ and $q$ are even or
both $p$ and $q$ are odd, and $d_{(p,q)}$ is the disk of radius $1$ centered at the
point $(p,q)$.  Consider the ideal hyperbolic polyhedron with infinitely many
vertices, $\PP_{\infty}$, corresponding to
$D_{\infty}$.  This polyhedron has all dihedral angles equal to $\pi/2$.  Applying
the decomposition into tetrahedra described in the proof of
Proposition~\ref{upper2}, it is seen that the sum of the volumes of the tetrahedra
meeting each vertex is exactly $\frac{V_8}{2}$. A 
sequence of polyhedra, $\PP_{2k},$ which have volume-to-vertex ratio converging
to that of
$\PP_{\infty}$ will be constructed.    

For each even natural number $2k$, $k\geq3$, consider the set of lines in the
plane $L_{2k}=\{(x,y) \in \RR^2 \mid \,  y=0,\,  y=2k \text{, or }y=\pm x + z,\,
z \in \ZZ\}$.  Now let $\PP_{2k}$ be the hyperbolic polyhedron with $1$--skeleton
given by 
$$\PP_{2k}^{(1)}=\{(x,y) \in L_{2k} \mid \, 0\leq y \leq
2k\}/\{(x,y)\sim(x+2k, y)\}$$ 
and all right angles.  See figure~\ref{figure:p6} for an illustration of
$\PP_6$.  The existence of such a
hyperbolic polyhedron is guaranteed by Andreev's theorem.  Equivalently, there
is a simply connected rigid disk pattern, $D_{2k}$, in the plane with each disk
corresponding to a face and right angles between disks which correspond to
intersecting faces. The vertices and faces of $\PP_{2k}$ will be referred to
in terms of the $(x,y)$ coordinates of the corresponding vertices and faces of
$L_{2k}$. 

The polyhedra $\PP_{2k}$ will prove the proposition.  The volume of
$\PP_{2k}$ is expressed as the sum of volumes of
cones on faces and Lemma~\ref{vc} is used to analyze the limiting volume-to-vertex 
ratio.  

Choose coordinates for the upper half-space model of $\HH^3$ so that the vertex
$(0,0)$ of $\PP_{2k}$ is located at infinity.  Then the volume of $\PP_{2k}$ may
be written
$$\vol(\PP_{2k})=\sum_{d} \vol(C(p(d))),$$ where the sum is taken over all faces
$d$ which do not meet the vertex $(0,0)$.  Using Lemma~\ref{vc}, the volume of
each $C(p(d))$ can be estimated in terms of the number of generations of disks
surrounding $d$ which agree with $D_{\infty}$. 

Fix a disk $d_{\infty} \in D_{\infty}.$  For each $m\in \ZZ$, $m\geq 0$, define $F_m$ to
be the set of disks $d\in D_{2k}$ for which $(D_{2k},d)$ and $(D_{\infty},
d_{\infty})$
agree to generation $m$, but do not agree to generation $m+1$.  The $2k$ faces of
$\PP_{2k}$ centered at the points $(i+1/2, k)$, $0\leq i \leq 2k-1$ as well as
the $4k$ faces which share an edge with them except for the face centered at
$(0,k-1/2)$ make up $F_{k-1}$.  Thus  
$$|F_{k-1}|=6k-1.$$
The set $F_{k-2}$ consists of the face centered at $(0,k-1/2)$ along with the
faces centered at the $8k-1$ points with coordinates $(i + 1/2 , k \pm 1)$ and
$(i, k \pm 3/2)$ for $0 \leq i \leq 2k-1$, excluding the face centered at
$(0,k-3/2)$.  In general, for $2 \leq l \leq k-1$, $F_{k-l}$ consists of the
face centered at $(0, k-(2(l-1)-1)/2)$ along with the faces centered at the
$8k-1$ points with coordinates $(i+1/2, k\pm (l-1))$ and $(i, k \pm (2l-1)/2)$
for $0\leq i \leq 2k-1$, excluding the face centered at $(0, k-(2l-1)/2)$.
Hence for $2\leq l\leq k-1,$ 
$$|F_{k-l}| =8k.$$   See
figure~\ref{figure:p6} for an example.  

\begin{figure}
\labellist
\hair 2pt
\pinlabel $(0,0)$ [tr] at 96 352
\pinlabel $1$ at 104 604
\pinlabel $1$ at 152 604
\pinlabel $1$ at 210 604
\pinlabel $1$ at 264 604
\pinlabel $1$ at 320 604
\pinlabel $1$ at 376 604
\pinlabel $2$ at 104 548
\pinlabel $2$ at 152 548
\pinlabel $2$ at 210 548
\pinlabel $2$ at 264 548
\pinlabel $2$ at 320 548
\pinlabel $2$ at 376 548
\pinlabel $1$ at 104 490
\pinlabel $2$ at 152 490
\pinlabel $2$ at 210 490
\pinlabel $2$ at 264 490
\pinlabel $2$ at 320 490
\pinlabel $2$ at 376 490
\pinlabel $1$ at 152 436
\pinlabel $1$ at 210 436
\pinlabel $1$ at 264 436
\pinlabel $1$ at 320 436
\pinlabel $1$ at 376 436
\pinlabel $1$ at 124 576
\pinlabel $1$ at 182 576
\pinlabel $1$ at 236 576
\pinlabel $1$ at 292 576
\pinlabel $1$ at 346 576
\pinlabel $1$ at 404 576
\pinlabel $2$ at 124 520
\pinlabel $2$ at 182 520
\pinlabel $2$ at 236 520
\pinlabel $2$ at 292 520
\pinlabel $2$ at 346 520
\pinlabel $2$ at 404 520
\pinlabel $1$ at 124 464
\pinlabel $1$ at 182 464
\pinlabel $1$ at 236 464
\pinlabel $1$ at 292 464
\pinlabel $1$ at 346 464
\pinlabel $1$ at 404 464
\endlabellist
\begin{center}
\scalebox{.50}{\includegraphics{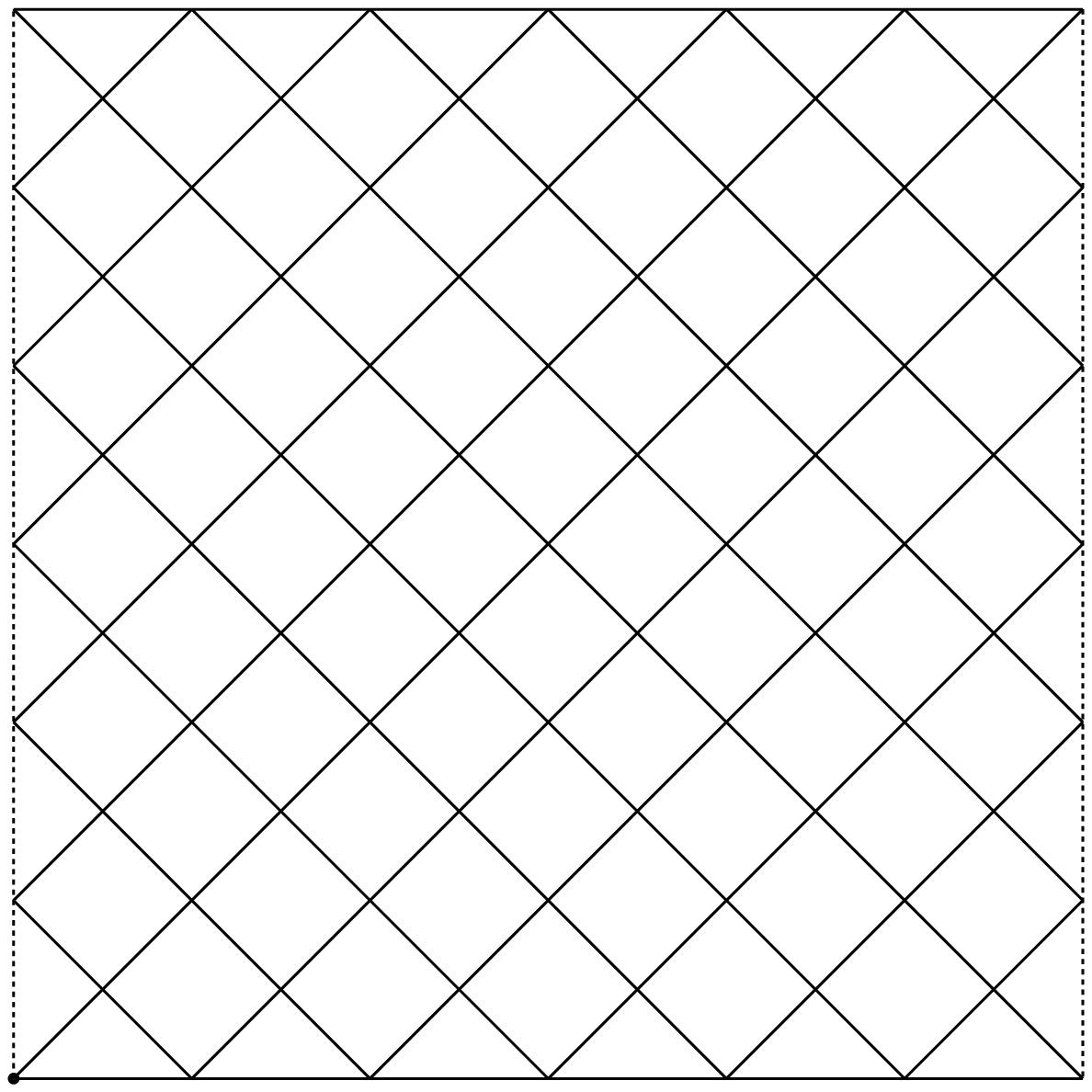}}
\end{center}
\caption{Identify the two vertical
sides to obtain $\PP_6$.  A face labeled by an integer $n$ has $n$ generations
of disks about it.  Unlabeled faces have $0$ generations about them.}
\label{figure:p6} \end{figure}

The polyhedron $\PP_{2k}$ has $8k^2+2k+2$ faces.  In the previous paragraph, it
was  
found that 
$$\left|\bigcup_{l=1}^{k-1} F_l\right|=8k^2-10k-1.$$
The remaining $12k+3$ faces consist of the following:  $4$ vertical faces which
do not contribute to the volume, one $2k$--gon, $4k-2$ triangular faces, and $8k$
rectangular faces in $F_0$.  The
maximum value of the Lobachevsky function, $\Lambda(\theta)$, is attained for
$\theta=\pi/6$ \cite{thurstonnotes}.  Hence, the formula for the volume of a
cone on an ideal polygon given in the proof of Lemma~\ref{vc} implies that the
volume of the cone on the $2k$--gon is less than or equal to $2k
\Lambda(\pi/6)$. Similarly, each of the remaining triangular and rectangular
faces have volume less than or equal to $4 \Lambda(\pi/6)$.  This implies that
the leftover faces have cone volume summing to a value $L\leq 14k\Lambda(\pi/6)$.  

The volume of the cone to infinity of any face
in $D_{\infty}$ is $4 \Lambda(\pi/4)=V_8 /2$.  By Lemma~\ref{vc}, there exists a
real-valued, positive function, $\delta_m$, on $F_m$ such that for each face $f \in F_m$,
$0 \leq \delta_m(f) \leq \epsilon_m$ and $\vol(C(f))=V_8/2 \pm \delta_m(f)$.
Therefore, the volume of $\PP_{2k}$ can written as:

$$\vol(\PP_{2k})=\sum_{l=1}^{k-1} \sum_{f\in F_{l}} \left(\frac{V_8}{2}
\pm \delta_{l}(f)\right)+L. $$
Using the analysis of the $F_m$ from
above, expand the sums and collect terms to get

$$\vol(\PP_{2k})= (8k^2-10k-1)\frac{V_8}{2} + \sum_{l=1}^{k-1} \sum_{f\in
F_{l}}(\pm \delta_{l} (f)) +L .$$

The polyhedron $\PP_{2k}$ has $N_{2k}=8k^2+2k$ vertices.  Therefore 
$$\lim_{k\to\infty} \frac{8k^2-10k-1}{N_{2k}}\frac{V_8}{2}=\frac{V_8}{2}.$$

It remains to show that the ratio of the last two summands to the number of
vertices converges to zero.
Set $\overline{\delta}_{l}=\max_{f\in F_l} \delta_l(f)$.  Then 
$$\lim_{k \to \infty}\frac{\left|\sum_{l=1}^{k-1} \sum_{f\in F_l} (\pm
\delta_l(f)) \right|}{N_{2k}}\leq
\lim_{k \to \infty} \left( \frac{(6k-1)\overline{\delta}_{k-1}}{N_{2k}}+
 \frac{8k\sum_{l=1}^{k-2} \overline{\delta_l}}{N_{2k}}\right)=0$$
because $\overline{\delta}_l \to 0$ as $l \to \infty$. 
Also since $L<14k \Lambda(\pi/6)$, 
$$\lim_{k\to\infty}\frac{L}{N_{2k}}=0.$$  

Therefore, 
$$\lim_{k \to \infty} \frac{\vol(\PP_{2k})}{N_{2k}}= \frac{V_8}{2}.$$
\end{proof}

The argument to prove Proposition~\ref{asy2} easily adapts to prove the
following:

\begin{proposition}\label{asy3}
There exists a sequence of ideal $\pi/3$--equiangular polyhedra $\PP_i$ with $N_i$ vertices such that
$$\lim_{i\to \infty} \frac {\vol(\PP_i)}{N_i} = \frac{3 V_3}{2}.$$
\end{proposition}
\begin{proof}
Consider the regular hexagon $H$ in the plane formed by the vertices $(0,0)$,
$(1,0)$, $(3/2, \sqrt{3}/2)$, $(1,\sqrt{3})$, $(0,\sqrt{3})$, and
$(-1/2,\sqrt{3}/2)$.  Let $G$ be the lattice of translations generated by
$\{(x,y)\mapsto (x+3,y) \}$ and $\{(x,y) \mapsto (x+3/2,y+ \frac{\sqrt{3}}{2})$.
Now define $T$ to be the orbit of the hexagon $H$ under the action of $G$ on the
plane.  This orbit is a tiling of the plane by regular hexagons.  As in the
previous construction, define 
$$L_{2k}=\{(x,y)\in \RR^2 \mid \, y=0, \,
y=2k\sqrt{3} \text{, or } (x,y) \text{ lies on a vertex or edge of }T\}.$$
 Let
$\QQ_{2k}$ be the hyperbolic polyhedron with $1$--skeleton
$$\QQ_{2k}^{(1)}=\{(x,y)\in L_{2k} \mid \, 0\leq y \leq 2k\sqrt{3}\}/\{(x,y)
\sim (x+3k,y)\}$$
 and all dihedral angles $\pi/3$. This is a polyhedron with
$4k^2+k+2$ faces and $N_{2k}= 8k^2 +2k$ vertices. Let $D_{2k}$ be the associated
simply connected rigid disk pattern and $D_{\infty}$ to be the infinite circle
pattern with $G(D_{\infty})$ equal to a tiling of the plane by equilateral
triangles with each edge labeled $\pi/3$.

\begin{figure} 
\labellist
\hair 2pt
\pinlabel $(0,0)$ [t] at 40 410
\pinlabel $(1,0)$ [t] at 80 410
\pinlabel $1$ at 60 713
\pinlabel $1$ at 176 713
\pinlabel $1$ at 292 713
\pinlabel $2$ at 60 646
\pinlabel $2$ at 176 646
\pinlabel $2$ at 292 646
\pinlabel $1$ at 60 580
\pinlabel $2$ at 176 580
\pinlabel $2$ at 292 580
\pinlabel $1$ at 176 514
\pinlabel $1$ at 292 514
\pinlabel $1$ at 118 680
\pinlabel $1$ at 234 680
\pinlabel $1$ at 350 680
\pinlabel $2$ at 118 612
\pinlabel $2$ at 234 612
\pinlabel $2$ at 350 612
\pinlabel $1$ at 118 546
\pinlabel $1$ at 234 546
\pinlabel $1$ at 350 546
\endlabellist
\begin{center}
\scalebox{.50}{\includegraphics{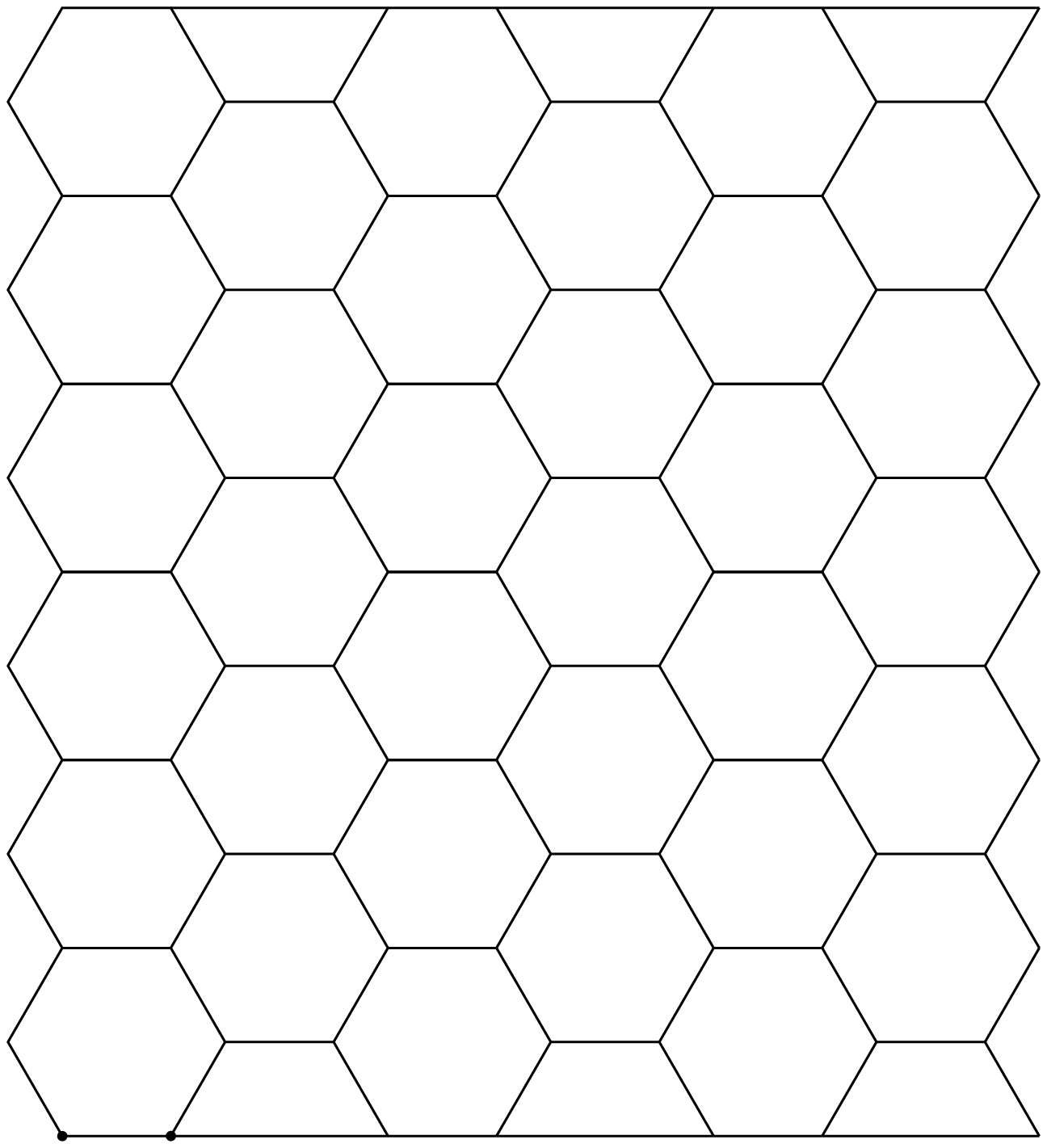}} 
\end{center}
\caption{Identify the left and right sides to get $\QQ_6$.  The labeling is as in
the previous figure.} 
\label{figure:hex} 
\end{figure}

As in the proof of Proposition~\ref{asy2}, choose coordinates so that the vertex at $(0,0)$ is at infinity.
Recall that for a fixed choice of $d_{\infty} \in D_{\infty}$, $F_m$ is defined to be the set of disks $d \in D_{2k}$ for which
$(D_{2k},d)$ and $(D_{\infty},d_{\infty})$ agree to generation $m$, but do not
agree to generation $m+1$.   The set $F_{k-1}$ consists of $3k-1$ faces, while the remaining $F_{k-l}$
for $2 \leq l \leq k-1$ consist of $4k$ faces.  Again, the faces not contained
in $F_{k-l}$ for $1 \leq l \leq k-1$ have cone volume summing to a value
$L$ bounded
above by a constant multiple of $k$ where the bound is independent of $k$.  The
volume of the cone to infinity of any face in the regular hexagonal circle
pattern is $6 \Lambda(\pi/6)=3V_3$.  There exists a function, $\delta_{l}$,
with the same properties as above.  The volume of $\QQ_{2k}$ is
$$\vol(\QQ_{2k})=
\sum_{l=1}^{k-1} \sum_{f\in F_{l}} \left(3V_3 \pm
\delta_{l}(f)\right)+ L.$$  The argument finishes exactly as the all
right-angled case to give  $$\lim_{k \to \infty}
\frac{\vol(\QQ_{2k})}{N_{2k}}=\frac{3V_3}{2}.$$
\end{proof}

To finish the proof of Theorem~\ref{pi2finite}, a minor modification to
Lemma~\ref{vc} is made.
 
Recall that for a disk pattern $D$, $G(D)$ is the graph with a vertex for each
disk and an edge connecting two vertices which have corresponding disks with
non-empty interior intersection.  Suppose that $D$ is a non-ideal disk pattern
and that $c$ is a disk which intersects $l$ neighboring disks, $d_{1}, \dots
d_{l}.$  The intersection of $S(c)$ with each of the $S(d_{i})$ is a finite
length geodesic segment.  The union of the $l$ geodesic segments along with the
disk bounded by them in $S(c)$ is a polygon $p(c)$.  Let $x_0\in \HH^3$ be a
point which is not contained in $S(c)$.  Denote by $C(p(c),x_0)$ the
cone of $p(c)$ to the point $x_0$.  The cone to the point at infinity in the
upper half-space model of $\HH^3$ will be denoted by $C(p(c),\{\infty\})$.

Let $c$ be a disk in a simply connected, non-ideal, finite disk pattern, $D$,
with associated polyhedron $\PP$. Suppose $c_{max}$ realizes the quantity
$$\max_{c'} d_{G(D)} ( c,c').$$
Define a \textit{cut point for $c$} to be any vertex of $p(c_{max})$.

\begin{lemma}\label{vcc}
Suppose that $D_n$ and $D_{\infty}$ are simply connected, non-ideal, rigid, disk
patterns, such that $(D_n,c_n)$ and $(D_{\infty},c_{\infty})$ satisfy
Proposition~\ref{hrs} and $D_{\infty}$ fills the entire plane.  Suppose also that $x_n$ is a cut point for $c_n$ in
$D_n$. 
Then
  $$\lim_{n\to \infty}
\vol(C(p(c_n),x_n))=\vol(C(p(c_{\infty}),\{\infty\})).$$ 
Moreover, there exists a
bounded sequence $0\leq \epsilon_n \leq K <\infty$ converging to zero such that
$| \vol(C(p(c_n),x_n))-\vol(C(p(c_{\infty}),\{\infty\})) | \leq \epsilon_n$ 
\end{lemma}

\begin{proof}
Note that for any choice of points $y_n \in S(c_n)$, $d(x_n , y_n) \to
\infty$ as $n \to \infty$. Also, since $D_{\infty}$ fills the entire plane, the
distance measured in $G(D_n)$ from $c_n$ to a disk $c'_n$ such that $S(c'_n)$
contains $x_n$ goes to infinity.  Hence as $n$ goes to infinity, $x_n$
approaches the point at infinity, so the compact cone $C(p(c_n), x_n)$
approaches the infinite cone $C(p(c_{\infty}), \{\infty\})$.  Therefore it
suffices to show that the volume of the compact cones approaches that of the
infinite cone.

A \textit{$3$--dimensional hyperbolic orthoscheme} is a hyperbolic tetrahedron
with a sequence of three edges $v_0v_1$, $v_1v_2$, and $v_2v_3$ such that
$v_0v_1 \bot v_1v_2 \bot v_2v_3$.  See figure~\ref{figure:comortho}.
Suppose that the degree of $p(c_n)$ is $a_n$.  The cone, $C(p(c_n),x_n)$, can be
decomposed into $2a_n$ orthoschemes by the procedure described at the beginning
of section~\ref{S:upper}.  
\begin{figure} 
\labellist
\small\hair 2pt
\pinlabel $\alpha_n$ [bl] at 118 137
\pinlabel $\beta_n$ [br] at 41 136
\pinlabel $\gamma_n$ [tr] at 39 48
\pinlabel $v_0$ [t] at 90 0
\pinlabel $v_1$ [r] at 0 92
\pinlabel $v_2$ [b] at 90 181
\pinlabel $v_3$ [l] at 143 92
\endlabellist
\begin{center}
\scalebox{.80}{\includegraphics{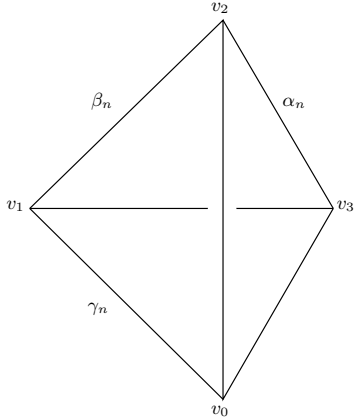}} 
\end{center}
\caption{A compact orthoscheme, $T(\alpha_n, \beta_n, \gamma_n)$.  The unlabeled edges have dihedral angle $\pi/2$} 
\label{figure:comortho} 
\end{figure}

The volume of one of the compact orthoschemes, $T(\alpha_n,
\beta_n, \gamma_n)$ determined by
angles $\alpha_n,$ $\beta_n$, and $\gamma_n,$ as shown in 
figure~\ref{figure:comortho},
is given by
\begin{align*}
\vol(T(\alpha_n, \beta_n, \gamma_n) = \frac{1}{4} &\bigg(
\Lambda(\alpha_n  +\delta_n) -
\Lambda(\alpha_n-\delta_n)+\Lambda(\gamma_n+\delta_n)  
- \Lambda(\gamma_n-\delta_n)  \\
&-  \Lambda \left(\frac{\pi}{2} -\beta_n + \delta_n
\right) 
+\Lambda \left(\frac{\pi}{2} -\beta_n -
\delta_n \right) + 2 \Lambda \left(\frac{\pi}{2}-\delta_n \right) \bigg) ,  
\end{align*}
where
$$0 \leq \delta_n = \arctan{\frac{\sqrt{-\Delta_n}}{\cos{\alpha_n} \cos{\gamma_n}}} <
\frac{\pi}{2},$$
and 
$$\Delta_n= \sin^2{\alpha_n} \sin^2{\gamma_n} - \cos^2{\beta_n}.$$ 
This is due to Lobachevsky.  See, for example, page 125 of \cite{geomii}.  

Similarly, the cone $C(p(c_{\infty}), \{\infty\})$ can be decomposed into
orthoschemes of the form $T(\alpha_{\infty},
\pi/2-\alpha_{\infty}, \gamma_{\infty})$ with one ideal vertex .  The volume of this orthoscheme
is given by
$$\vol(T(\alpha_{\infty}, \pi/2 - \alpha_{\infty}, \gamma_{\infty}) = \frac{1}{4}
\left( \Lambda(\alpha_{\infty} +
\gamma_{\infty}) + \Lambda( \alpha_{\infty} - \gamma_{\infty}) + 2 \Lambda( \pi/2
- \alpha_{\infty}) \right).$$  

As $n \to \infty$, $\Delta_n \to -\sin^2{\alpha_n} \cos^2{\gamma_n}$, so
$\delta_n \to \alpha_n$.  Therefore the sequence of volumes of the compact
orthoschemes converges to that of the orthoscheme with one ideal vertex.
Summing over all orthoschemes in the decomposition proves the lemma.
\end{proof}

The next proposition will complete the proof of Theorem~\ref{pi2finite}.
 \begin{proposition}\label{asycom2}
There exists a sequence of compact $\pi/2$--equiangular polyhedra $\PP_i$ with $N_i$ vertices such that
$$\lim_{i\to \infty} \frac {\vol(\PP_i)}{N_i} = \frac{5 V_3}{8}.$$
\end{proposition}

\begin{proof}
Define $L'_{2k}$ to be $L_{2k}$ as in the proof of Proposition~\ref{asy3} along
with the tripods as shown in figure~\ref{figure:hexmod}.  The tripods must be added
to remove the degree $4$ faces.
\begin{figure} 
\labellist
\hair 2pt
\pinlabel $(0,0)$ [tr] at 20 2		
\pinlabel $(1,0)$ [t] at 58 2		
\pinlabel $1$ at 97 270
\pinlabel $1$ at 213 270
\pinlabel $1$ at 330 270
\pinlabel $2$ at 97 204
\pinlabel $2$ at 213 204
\pinlabel $2$ at 330 204
\pinlabel $1$ at 97 137
\pinlabel $1$ at 213 137
\pinlabel $1$ at 330 137
\pinlabel $2$ at 40 237
\pinlabel $2$ at 155 237
\pinlabel $2$ at 272 237
\pinlabel $1$ at 40 170
\pinlabel $2$ at 155 170
\pinlabel $2$ at 272 170
\endlabellist
\begin{center}
\scalebox{.50}{\includegraphics{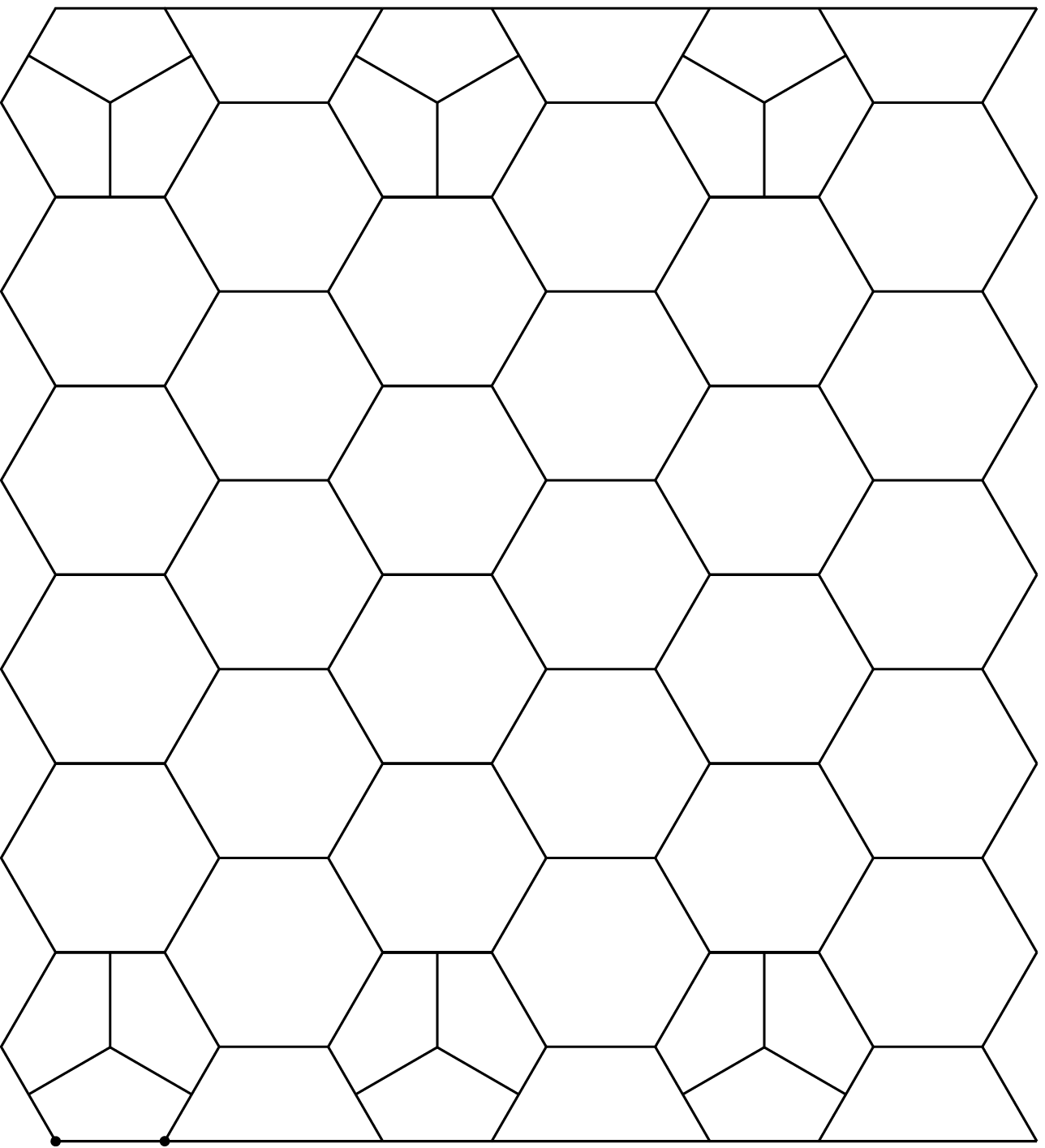}} 
\end{center}
\caption{Identify the left and right sides to get $\RRR_6$.  The labeling is as in
the previous figures.} 
\label{figure:hexmod} 
\end{figure}
 Let $\RRR_{2k}$ be
the polyhedron with $1$--skeleton 
$$\RRR_{2k}^{(1)}=\{(x,y)\in L'_{2k} \mid \, 0\leq y \leq 2k\sqrt{3}\}/\{(x,y)
\sim (x+3k,y)\}$$
and all dihedral angles equal to $\pi/2$.  For each $k>2$, this can be realized
as a compact hyperbolic polyhedron by Andreev's theorem.  The rest of the proof of this
proposition mirrors the proof of Proposition~\ref{asy3} exactly, using Lemma~\ref{vcc} in
place of Lemma~\ref{vc}.
\end{proof}
\bibliographystyle{./hamsplain}
\bibliography{./volpoly}
\end{document}